\documentclass[reqno, 11pt]{amsart}
\usepackage{ifthen}

\usepackage{setspace}
\usepackage[a4paper, left=3cm, right=3cm, top=4cm]{geometry}
\usepackage{amssymb, amsthm, epsf, epsfig}
\usepackage{amsmath}
\usepackage{verbatim}
\usepackage{caption}
\usepackage{subcaption}
\usepackage{dsfont}
\usepackage{wrapfig}
\usepackage[shortlabels]{enumitem}
\usepackage[usenames,dvipsnames]{xcolor}
\usepackage{esvect}

\newtheorem{theorem}{Theorem}
\newtheorem{proposition}{Proposition}
\newtheorem{lemma}{Lemma}

\theoremstyle{definition}
\newtheorem{definition}{Definition}
\newtheorem{example}{Example}

\theoremstyle{remark}

\newcommand{\C}{\mathbb{C}}
\newcommand{\E}{\mathbb{E}}
\newcommand{\PP}{\mathbb{P}}
\newcommand{\V}{{\mathbb{V}\mbox{ar}}}

\newcommand{\map}[1]{\textbf{#1}}

\numberwithin{equation}{section}

\title{Asymptotic normality of pattern counts in random maps II}

\author{Eva-Maria Hainzl}
\thanks{This work was partially founded by the Austrian Science
Foundation FWF, projects P 35016, F 100203 and the French project ANR LOUCCOUM.}

\begin{document}

\maketitle 
\begin{abstract}
    In a recent work, a central limit theorem for pattern counts in random planar maps was proven by reducing the problem to a face count problem~\cite{patternocc}. We provide a shorter proof by circumventing this reduction through the computation of bivariate coefficient asymptotics from a functional equation with one catalytic variable and extend the result to pattern counts with arbitrary boundary and new map classes. 
\end{abstract}


\section{Introduction}

Pattern counts in random planar maps have been of interest since the early 90s~\cite{BenderGaoRichmond,GaoWormald,GaoWormald2,DrmotaStufler,Yu}. In a recent work~\cite{patternocc}, the authors proved a Gaussian central limit theorem for pattern counts in random planar maps for patterns with simple boundary. The main results in the following is an extension of limit theorem to counts of patterns with arbitrary boundary and analogous limit theorems in bipartite and 2-connected maps. In particular, the latter class of maps was not amenable to the original proof strategy that consisted of reducing the pattern count problem to a face count problem because the reduction to face counts can create cut vertices in the original 2-connected map.\\
We give a direct proof that avoids the reduction to face counts through bivariate coefficient extraction from catalytic variable equations.  
We lay out the details of this simplified proof along planar maps in the next section and subsequently prove an extension of the central limit theorem in bipartite and 2-connected maps.

The foundation of our method lies in the recursive enumeration of rooted planar maps, which dates back to the seminal work of Tutte in the 1960s. His decomposition of planar maps leads to functional equations with a catalytic variable for the corresponding generating functions. For instance, if $M(z, u)$ denotes the generating function where $z$ marks edges and $u$ marks the valency of the root face, the classical Tutte equation (\cite{tutte_1963}) reads:
\[
    M(z,u) = 1 + zu^2 M(z, u)^2 + zu \cdot \frac{M(z, u) - M(z, 1)}{u - 1}.
\]
This formulation can be enriched by additional variables to count specific substructures. For example, by introducing a variable $x$ to mark non-root faces with valency~$2$, we obtain:
\[
    M(z, u, x) = 1 + zu^2 M(z, u, x)^2 + zu  \frac{M(z, u, x) - M(z, 1, x)}{u - 1} + z(x-1)(M(z, u, x) - 1).
\]

In this framework, Drmota and Panagiotou \cite{DrPa} established central limit theorems for faces of valency $k$, while Yu~\cite{Yu} handled the case of simple $k$-gons which are valency $k$ faces with a boundary consisting of $k$ distinct vertices and $k$ distinct edges. However, a key limitation of these approaches applied to more general patterns is that they break down when patterns are allowed to self-intersect: no single equation using a counting parameter can in general describe overlapping pattern occurrences. Drmota, Hainzl and Wormald~\cite{patternocc} circumvented this problem through a theorem by Gao and Wormald~\cite{GaoWormald}, leveraging the asymptotics of high moments of face counts. We simplify this idea and compute directly the asymptotics of the high moments of the pattern counts.\\

\section{A simplified proof of asymptotic normality of pattern counts}

We start by introducing some basic definitions and notations.
\begin{definition}[Rooted planar maps]
    A \emph{planar map} is a connected planar graph (with loops and multiple edges allowed) embedded onto the sphere. If one of its edges is oriented, we call it \emph{rooted}. The oriented edge is further called the \emph{root edge} and the vertex from which the root edge is pointing away the \emph{root vertex}.
    A planar map separates the surface into several connected regions called \emph{faces}. The face to the left of the root edge is called the \emph{root face} or the \emph{exterior face}.
    The \emph{valency} of a face is the number of edges incident to it, bridges being counted twice. A face of valency $m$ is called an \emph{$m$-gon}.
    We define the \emph{boundary} of a rooted map as the set of all edges and vertices incident with the root face.

    We count rooted maps up to root-preserving isomorphism according to their number of edges and we denote the generating function of planar maps by
        \[
            M(z,u) = \sum_{n,j\geq 0} m_{j,n}u^jz^n,
        \]
    where $z$ marks the number of edges and $u$ the root face valency.
\end{definition}

Now we are interested in the number of pattern occurrences of a fixed map $\map{p}$. Even though the term seems self-explanatory, formally, a pattern occurrence is defined as follows.

\begin{definition}[Pattern occurrences~\cite{BenderGaoRichmond} ]\label{def:submaps}
    Let~\map{p} be a rooted map. We say that~\map{p} occurs as a pattern in a
    map~\map{m} if~\map{m} can be obtained by extending~\map{p} in the following way:
    \begin{enumerate}
        \item[(a)] adding vertices to the interior of the root face of~\map{p},
        \item[(b)] adding edges with their endpoints being either vertices or edges from the
        boundary of~\map{p} or newly created vertices,
        \item[(c)] rerooting the so obtained map in such a way that its new root face is not contained in an interior face of~\map{p}.
    \end{enumerate}
\end{definition}

The main result for pattern counts in random maps is an extension of Theorem 1 in \cite{patternocc}.

\begin{theorem}\label{thm:clt}
    Let $\map{p}$ be a planar map and $X_n$ the number of pattern occurrences of $\map{p}$ in a random planar map of size $n$. Then,
    \[
        \frac{X_n-\mu_n}{\sigma_n} \rightarrow \mathcal{N}(0,1)
    \]
    where $\mu_n = c_1n$ and $\sigma_n = c_2\sqrt{n}$ for some computable constants $c_1,c_2 \geq 0$ as $n\rightarrow \infty$.
\end{theorem}

Even though the proof of this theorem almost immediately follows from Lemma~\ref{lem:S1} below, we present a simplified proof that naturally extends to further map classes which where not covered by the proof strategy in \cite{patternocc} (such as 2-connected maps).\\

The proof consists of two main ingredients. The first is the theory of universal asymptotic behavior of solutions to discrete differential equations with one catalytic variable (see Theorem~\ref{thm:dny}) as proven in~\cite{DrmotaNoyYu}. The second is a powerful theorem by Gao and Wormald~\cite{GaoWormald}, which provides conditions under which asymptotic normality follows from the asymptotics of factorial moments (see Theorem~\ref{thm:gao}). Both have been used independently to prove central limit theorems for counts of patterns with cannot self-intersect.

However, if the pattern has a boundary with a pinch point or if pattern occurrences can overlap in numerous ways, it is unclear how to capture this overlapping behavior recursively. A key to resolve this issue is considering high moments of the number of pattern occurrences $X_n$ of a fixed map \map{p} in a random map of size $n$. The following theorem leverages moments of order $k = \Theta(\sqrt{n})$ and entails a central limit theorem.

\begin{theorem}[Gao and Wormald, \cite{GaoWormald}]\label{thm:gao}
    Let $X_n$ be a sequence of non-negative integer-valued random variables with sequences $\mu_n \to \infty$ and $\sigma_n > 0$ satisfying
    \[
        \sigma_n \log^2 \sigma_n = o(\mu_n), \quad \mu_n = o(\sigma_n^3).
    \]
    Suppose that
    \[
        \mathbb{E}[(X_n)_k] \sim \mu_n^k \exp\left( \frac{k^2}{2} \cdot \frac{\sigma_n^2 - \mu_n}{\mu_n^2} \right)
    \]
    uniformly for $c\mu_n/\sigma_n \le k \le c'\mu_n/\sigma_n$, for constants $0 < c < c'$. Then,
    \[
        \frac{X_n - \mu_n}{\sigma_n} \xrightarrow{d} \mathcal{N}(0, 1).
    \]
\end{theorem}

Now, let us point out that the $k$-th factorial moment of $X_n$ equals
\[
    \mathbb{E}[(X_n)_k] = \sum_{\ell \ge k} \ell (\ell-1)\cdots (\ell-k+1) \cdot \frac{m_{n,\ell}}{m_n},
\]
where $m_n$ denotes the number of planar maps with $n$ edges, and $b_{n,\ell}$ the number of maps with $\ell$ occurrences of the pattern. In particular, the factor $\ell (\ell-1)\cdots (\ell-k+1)\,b_{n,\ell}$ counts planar maps with $\ell$ pattern occurrences, among which $k$ are labeled. Summing over all $\ell$ just gives the number of maps with $k$ labeled pattern occurrences (among arbitrary many). A key observation is the following lemma which holds for several map classes.

\begin{lemma}\label{lem:S1}
    Let $\mathcal{C}$ be a (sub)set of planar maps and \map{p} be a planar map. Let further $\mu_n$ the expected number of occurrences of \map{p} in a random map or size $n$ in $\mathcal{C}$, $\sigma_n$ be the standard deviation thereof and let $m^\circ_{n,k}$ be the number of all maps of size $n$ with $k$ labeled occurrences of \map{p} in $\mathcal{C}$ and let $m^{\circ,\times}_{n,k}$ be the number of all such maps where each labeled pattern occurrence intersects at most one other labeled occurrence of \map{p}. If \map{p} has a simple boundary, $\mu_n, \sigma_n^2 \in \Theta(n)$ and $k = \Theta(\sqrt{n})$ then, as $n\rightarrow \infty$, 
    \begin{equation}\label{eq:sandw}
        m_{n,k}^{\circ,\times} \leq m_{n,k}^\circ \leq  m_{n,k}^{\circ,\times} + \left(\frac{\mu_n}{2}\right)^k m_{n} + o\left(m_{n,k}\right). 
    \end{equation}
    If \map{p} does not have a simple boundary but the maximal vertex degree $\Delta_n$ in a random map with $n$ edges in $\mathcal{C}$ satisfies
    \[
        \PP\left(\Delta_n > \xi(n)\right) \leq e^{-\delta \xi_n}
    \]
    for $\xi_n = \omega(\log n)$, then \eqref{eq:sandw} holds as well.
\end{lemma}

The proof of this lemma can be found in Section~\ref{sec:proofs}. 

It is known that the maximal vertex degree in random planar maps has expectation of order $\log(n)$ and its distribution has exponential tails for $\Delta_n = \omega(\log n)$, see \cite{max_deg}. 

So using this lemma, it is enough to enumerate maps with labeled pattern occurrences, where each labeled pattern intersects at most one other labeled pattern occurrence. Since there are only a finite number of ways for two pattern occurrences to intersect, we can introduce an additional counting variable $x$ where $x$ marks distinguished labeled pattern occurrences which intersect at most pairwise.

In order to do this, we list all \emph{intersection types} of the pattern. These describe all submaps containing two distinguished pattern occurrences.

\begin{definition}[Intersection types, rotations and deep faces]    
   Two pattern occurrences are said to \emph{intersect} if either they contain a common face, or they share a pinch point on their boundary such that the cyclic order of the incident components at the pinch point cannot be partitioned into two contiguous blocks, each containing only components of one occurrence.

    More precisely, suppose that removing the pinch point decomposes the first occurrence into components $B_1,B_2,B_3$ and the second occurrence into components $B'_1,B'_2$. Let the cyclic order of these components around the pinch point be the radial order induced by the embedding. The two occurrences intersect if this cyclic order is not of the form
    \[
        B_i,B_j,B_k,B'_\ell,B'_m
    \]
    for any permutation $i,j,k \in \{1,2,3\}$ and $\ell,m \in \{1,2\}$, so the components of the 
    two occurrences cannot be separated into two consecutive runs.
   
    Let $\map{p}$ and $\map{r}$ be rooted planar maps and let $\map{r}$ contain two distinguished 
    pattern occurrences $p_1$ and $p_2$ of $\map{p}$ which intersect. Further, suppose that each edge and vertex of $\map{r}$ is an edge or vertex of $p_1$ or $p_2$. Then $\map{r}$ is called a \emph{rooted intersection type} of the pattern $\map{p}$.  
   
    \noindent    If  a rooted  map  $\map{n}$ can be obtained by rerooting $\map{r}$ such that the root face is preserved, then $\map{n}$ is called a {\em rotation} of $\map{r}$. 
        
    \noindent    An \emph{intersection type} $i$ of $\map{p}$ is the set of all rotations of a rooted intersection type $\map{r}$ of $\map{p}$. We further denote the number of pairwise non-isomorphic rotations of $\map{r}$ by $r_i$, where isomorphisms must preserve $\{p_1,p_2\}$. That is, the intersection type $i$ contains $r_i$ distinct rooted maps with an unordered pair of distinguished occurrences of $\map{p}$. 
    
    \noindent Given two intersecting occurrences $p_1$ and $p_2$  of the pattern $\map{p}$ in an arbitrary map $\map{m}$, we can obtain a rooted intersection type of $\map{p}$ by deleting all edges of $\map{m}$ except those in $\map{p}_1$ or  $\map{p}_2$, and assigning a root on the boundary of the (new) face that contains the root face of $\map{m}$. The intersection type does not depend on which root was chosen on the boundary. If it is type $i$, we say that {\em the intersecting pair of occurrences, $p_1$ and $p_2$,  has intersection type~$i$}.
    Note  that  a rooted intersection type can have interior faces that are not interior to either of the two distinguished occurrences of $\map{p}$. We call these {\em deep} faces.
\end{definition}
Given the intersection types of a pattern, we are able set up a functional equation counting maps with \emph{some} distinguished and labeled pattern occurrences which intersect at most pairwise. 

In particular, if the root edge is incident to an occurrence of a distinguished and labeled pattern occurrence or an intersecting pair of such pattern occurrences, we can isolate the submap by a simple boundary corresponding to the boundary of its intersection type. In case there are pinch points on the boundary, we simply split the pinch points into two separate vertices and obtain a simple boundary (see Figure~\ref{fig:eyes1} for an example). Then, we count the number of maps with so-called partial simple boundaries and subsequently attach a rotation of the intersection type along the boundary. Finally, we insert maps with simple boundaries in the deep faces which do not necessarily have simple boundaries. However, we can just glue vertices of the maps with simple boundary to match the exact shape of the deep face. 

 In Figure~\ref{fig:map_decom} the recursive decomposition is illustrated, where the sum runs over all intersection types $j$, $S_{i_j}$ denotes a map with simple boundary of length $i_j$ and $P_{k_j}$ a map with partial simple boundary of length $k_j$).

Maps with partial simple boundary have been introduced in ~\cite{Yu}, where they were treated in detail. We will define the once again but only discuss their generating functions for bipartite maps in the next section.

\begin{definition}[(Partial) simple boundaries]
    Let \map{m} be a map with root face valency $k > i$. If the first $i$ steps of the path along the boundary of \map{m} starting at its root vertex and in the direction determined by its root edge orientation consists of $i$ distinct edges and $i$ distinct vertices, we say the map has a \emph{partial simple boundary} of length $i$.\\
    If \map{m} has a partial simple boundary of length $i-1$ and root face valency $i$, then it has a \emph{simple boundary}.
\end{definition}

\begin{figure}
    \centering
    \includegraphics[width=0.85\linewidth]{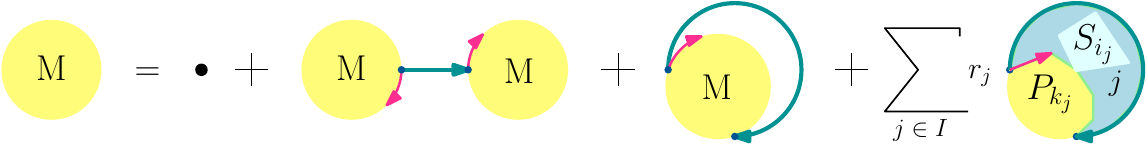}
    \caption{Decomposition of planar maps \emph{M} 
    with labeled patterns intersecting at most pairwise}
    \label{fig:map_decom}
\end{figure}

\begin{proposition}
    Let $\map{p}$ be a planar map with $e$ edges, a boundary of length $v$ and $r$ rotational symmetries. Further, we enumerate its intersection types by $1$ to $I$, and let $e_i$ be the number of edges, $v_i$ the root face valency, $d_{i,j}$ the number of deep $j$-faces and $r_i$ the number of rotations of intersection type $i \in [I]$. Then the generating function $M(z,u,x)$ for planar maps with some distinguished labelled pattern occurrences of~$\map{p}$, intersecting at most pairwise satisfies the equation
    \begin{align}
        M(z,u,x) &= 1+zuM(z,u,x)^2 + zu\frac{M(z,1,x)-uM(z,u,x)}{1-u} + x\frac{rz^{e-v+1}}{u^{v-2}}P^\mathcal{B}_{v-1}(z,u,x)\nonumber \\
        &\quad+ x^2\sum_{i=1}^I \frac{r_iz^{e_i-v_i+1}}{u^{v_i-2}}P_{v_i-1}(z,u,x)\prod_{j\geq 1}\left([u^j]P_{j-1}(z,u,x)\right)^{d_{i,j}}, \label{eq:bipartite}
    \end{align}
    where $z$ counts the number of edges, $u$ the root face valency, $x$ the number of distinguished, labeled occurrences of $\map{p}$ and $P_i(z,u,x)$ is the generating function of maps with partial simple boundary of length $i$.
\end{proposition}

\begin{proof}
    We use the classic decomposition scheme by Tutte from the 1960s \cite{tutte_1963} which is based on deleting the root edge. The root edge of the resulting map(s) is defined so as to preserve the root face and root vertex (see Figure \ref{fig:map_decom}). 
    In case the map disconnects, both components are assigned root edges, where in one component the root edge preserves the root vertex and the root face and in the other component the root vertex is the vertex where the deleted root edge pointed to. See Figure~\ref{fig:map_decom} for the decomposition scheme that we analyse. 

    The fourth term corresponds to the case where the root edge is incident to a labeled pattern occurrence which does not intersect with any other labeled pattern occurrence. In this case, we count the number of maps with partial simple boundary of length $v-1$ with $n-e+v-1$ edges and multiply it by $r$, the number of rotational symmetries of $\map{p}$. Gluing the map with partial simple boundary to one of these $r$ rotational symmetries of the pattern occurrence will decrease the root face valency of the map by $v-2$ and increase the number of edges by $e-v+1$. Note that the pattern occurrence does not necessarily have a simple boundary, but the gluing procedure is still bijective, since one can separate a pinch point that is incident to $j$ corners of the root face of $\map{p}$ into $j$ vertices that preserve the root face (see Figure~\ref{fig:eyes}), subsequently delete the interior edges of $\map{p}$ and finally the root edge and obtain a map with partial simple boundary.

    The same argument holds for the terms in the sum over all intersection types, where we additionally fill deep faces with maps with simple boundary and corresponding root face valency. These are enumerated by $[u^j]P_{j-1}(z,u,x), j\geq 0$ since a map with partial simple boundary of length $j-1$ and root face valency $j$ is indeed a map with simple boundary of length $j$. Again, one can argue that this procedure is also bijective for deep faces which do not have a simple boundary analogous to above.
\end{proof}

Once the functional equation for $M(z,u,x)$ is established, we aim to apply the following theorem providing universal asymptotic behavior of the coefficients in $z$. 

\begin{theorem}[Drmota, Noy, Yu~\cite{DrmotaNoyYu}]\label{thm:dny}
    Let $F(z,u,x)$ be the solution to the equation
    \[
        F(z,u,x) = Q\left(z,u,x,F(z,u,x),\frac{F(z,u,x)-F(z,1,x)}{u}\right),
    \]
    where $Q(z,u,x,y_0,y_1)$ is an analytic function around the origin and at the boundaries of convergence of $F(z,u,x), F(z,1,x)$ and let $Q(z,u,0,y_0,y_1)$ have non-negative Taylor coefficients. Let further $\rho(0)^{-1}>0$ be the only singularity at the radius of convergence of $F(z,1,0)$. Then
    \[
        [z^n]F(z,0,x) = c(x)n^{-\frac{5}{2}}\rho(x)^n\left(1+O\left(\frac{1}{n}\right)\right),
    \]
    uniformly for $x$ in a compact neighborhood of $x=0$, where $c(x)$ and $\rho(x)$ are analytic functions at $x=0$.
\end{theorem}

Clearly, to apply Theorem~\ref{thm:dny}, we need to check whether the generating functions describing maps with marked pattern occurrences incident to the root edge are indeed analytic functions in $z,u,x,M(z,u,x)$ and $M(z,1,x)$ and subsequently do a variable change $v=u-1$. Indeed, this was already proven for face counts in maps in~\cite{DrPa} and further used in~\cite{Yu} and~\cite{patternocc}. 

\begin{lemma}\label{lem:mov_bi}
	The generating functions $m_i(z,x):= [u^i]M(z,u,x)$ of planar maps with root face valency~$i$ are analytic functions in $z,x$ and $M(z,1,x)$ if $|z|\leq \frac{13}{25}$, $|M(z,1,x)-1| \leq \frac{5}{4}$ and $|x| \leq x_0$ for some $x_0>0$, small enough.
\end{lemma}

The proof of this statement is based on a fixed point argument on the space of $\ell^1(\C)$ sequences which we pick up in the proof of analogous statements for other map classes in the next section.

Now, we can use the following lemma to compute the higher moments of our random variable. 

\begin{lemma}[Bivariate asymptotics from catalytic equations]\label{lem:bivar}
    Let the coefficients of $F(z,0,x)$ in $z$ satisfy uniformly for $x$ in a compact neighborhood around $x=0$ the asymptotic formula
    \[
        [z^n]F(z,0,x) = c(x)n^{-\frac{5}{2}}\rho(x)^n\left(1+O\left(\frac{1}{n}\right)\right)
    \]
    where $c(x)$ and $\rho(x)$ are analytic functions at $x=0$. Then for, $k = \Theta(\sqrt{n})$,
    \[
        k![z^nx^k]F(z,0,x) \sim c_0n^{k-\frac{5}{2}}\rho(0)^n\left(\frac{\rho'(0)}{\rho(0)}\right)^ke^{\frac{k^2}{2n}\left(\frac{\rho''(0)\rho(0)}{\rho'(0)^2}-1\right)},
    \]
    as $n\rightarrow \infty$.
\end{lemma}

The proof of this lemma is based on the results in~\cite{DrmotaNoyYu} and computations using the saddle point method. It can be found in Section~\ref{sec:proofs}.

The moments of the number of pattern occurrences in random maps as required by Theorem~\ref{thm:gao} are then given by Lemma~\ref{lem:S1} and simple calculations based on Lemma~\ref{lem:bivar} which we omit to save space leading to the following expressions.

\begin{lemma}\label{lem:moments}
    Let $\map{p}$ be a planar map and $X_n$ the number of pattern occurrences of $\map{p}$ in a random map of size $n$. Then,
    \begin{align*}
        \E(X_n)&\sim \frac{\rho'(0)}{\rho(0)}n\\
        \V(X_n)&\sim \frac{\rho''(0)\rho(0)+\rho'(0)\rho(0)-\rho'(0)^2}{\rho(0)^2}n\\
        \E((X_n)_k)&\sim \left(\frac{\rho'(0)}{\rho(0)}n\right)^{k}e^{\frac{k^2}{2n}\left(\frac{\rho''(0)\rho(0)}{\rho'(0)^2}-1\right)}, &k = \Theta(\sqrt{n}),
    \end{align*}
    as $n \rightarrow \infty$.
\end{lemma}

Finally, we derive from Theorem~\ref{thm:gao} and Lemma~\ref{lem:moments} the desired central limit theorem.

\begin{example}[Flies]
    We consider the example of two 2-gons meeting at a vertex in Figure~\ref{fig:eyes1}, which we call a fly. The figure also illustrates, that a map with a root edge incident to a distinguished and labeled pattern occurrence can be decomposed into the root edge and a map with partial simple boundary of length $3$ which we glue at two vertices on the boundary path.
    
    Its six intersection types are listed in Figure~\ref{fig:eyes}. Note that the fact that the pattern occurrences are distinguished are relevant to the number of rotations of an intersection type. For example, Type 2 has two distinct rotations if the pattern occurrences were not distinguished. But in this case, the face that is shared by both pattern occurrences is distinguishable from the other two and breaks the rotational symmetry such that there are six possible rotations of the intersection type.
    \begin{align*}
        B(z,u,x) &= 1+zuB(z,u,x)^2+z\Delta B(z,u,x) + 2xP_3(z,u,x) +3x^2P_5(z,u,x)\\
        &+x^2\left(2P_7(z,u,x)+6P_5(z,u,x)(1+z+[u]P_1(z,u,x)\right)\\
        &+x^2\left(P_3(z,u,x)(4(z+[u]P_{1}(z,u,x))+2(z+[u]P_{1}(z,u,x))^2\right)
    \end{align*}
    Now, Theorem~\ref{thm:clt} immediately tells us that the number $X_n$ of flies in a random planar map satisfies a Gaussian central limit theorem.
\end{example}

\begin{figure}
    \centering
    \includegraphics[width=0.45\linewidth, page=1]{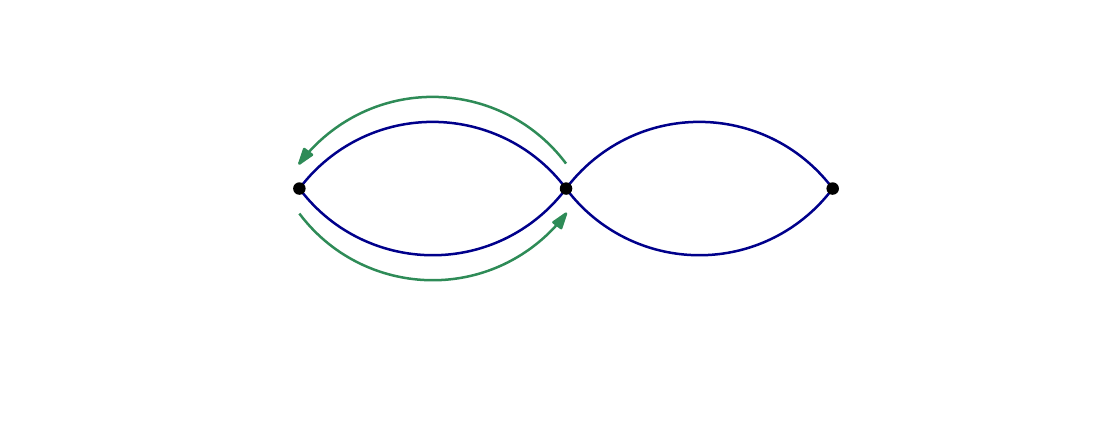} \includegraphics[width=0.45\linewidth, page=2]{eyes.pdf}\\\vspace{-5mm}
    \caption{Left: a fly. The green arrows correspond to the rotations of the pattern occurrence.\\
    Right: Decomposition of map where the root edge is incident to a distinguished fly into the root edge and a map with simple partial boundary.}
    \label{fig:eyes1}
    \vspace{4mm}
    \includegraphics[width=0.48\linewidth, page=3]{eyes.pdf}
    \includegraphics[width=0.48\linewidth, page=4]{eyes.pdf}\hspace{-10mm}$\,$\\
    \caption{The seven intersection types of two flies. The green arrows correspond to the possible rotations of the intersection type. The yellow faces are deep faces. Deep faces of valency $2$ are allowed to contract to a single edge in this listing.}
    \label{fig:eyes}
\end{figure}

\section{Extension to further map classes}

In this section, we adapt the proof in the previous section to bipartite and 2-connected planar maps. The simplified proof in the previous section is particularly interesting for the latter class of maps, since the proof in~\cite{patternocc} would have required us to delete the interior edges of all marked patterns, which in turn could have created cut vertices. Therefore, the resulting map would not have been 2-connected and we could not have simply applied existing theory on face counts in a random 2-connected map. However, the simplified proof extends quite straightforwardly, since 2-connected maps also satisfy a classic root edge decomposition.

Further, the crucial technical lemmas and theorems (such as Theorem~\ref{thm:gao}, Lemmas~\ref{lem:moments} and~\ref{lem:S1}) are already proven. So, our focus for both map classes is to
\begin{enumerate}
    \item set up an adequate functional equation with an additional counting variable accounting for single labeled pattern occurrences and pairs of intersecting labeled pattern occurrences.
    \item prove that we may apply Lemma~\ref{lem:S1} to the functional equation and its solution(s). In particular that the coefficients satisfy asymptotic formulas of the form
    \[
        c(x)n^{-\frac{5}{2}}\rho(x)^n
    \]
    where $c(x)$ and $\rho(x)$ are suitable analytic functions around $x=0$.
\end{enumerate}

Analogously to the general case, we adopt the root edge decomposition to describe the generating functions with additional counting variables and set up functional equations for pattern counts accordingly. Finally, we use Lemma~\ref{lem:moments} to compute the asymptotics of expectation, variance and higher factorial moments of the pattern counts and derive a central limit theorem.

\subsection{Bipartite maps}

We start by formally defining bipartite maps.
\begin{definition}[Bipartite (planar) maps]
    A \emph{planar map} is  \emph{bipartite} if and only if there exists a proper $2$-(vertex-)coloring. This is equivalent to restricting the faces of the planar maps to have even degrees. We denote their generating function by
        \[
            B(z,u) = \sum_{n,j\geq 0} b_{j,n}u^jz^n,
        \]
    where $z$ marks the number of edges and $u$ the root face valency divided by $2$ to avoid any periodicities.
\end{definition}

The main result for pattern counts in random bipartite maps is the following theorem.

\begin{theorem}\label{thm:bipartite}
    Let $\map{p}$ be a planar bipartite map with simple boundary and $X_n$ the number of pattern occurrences of $\map{p}$ in a random bipartite map of size $n$. Then,
    \[
        \frac{X_n-\mu_n}{\sigma_n} \rightarrow \mathcal{N}(0,1)
    \]
    where $\mu_n = c_1n$ and $\sigma_n = c_2\sqrt{n}$ for some computable constants $c_1,c_2 \geq 0$ as $n\rightarrow \infty$.
\end{theorem}

We expect the theorem to hold for patterns without simple boundary as well. However, since the exponential tail bound on the maximal degree in bipartite maps does not seem to be available in the literature, we restrict the statement to patterns with simple boundary. 

As already lined out, our goal is to first set up a functional equation for the generating function of bipartite maps enriched with a variable counting distinguished labeled pattern occurrences which intersect at most pairwise. Subsequently, we prove that the equation satisfies the analytic assumptions of Lemma~\ref{lem:bivar} such that we can deduce the desired asymptotics of the high moments in Theorem~\ref{thm:gao}.

Again, we will need maps with partial simple boundary to properly decompose the maps. Their generating functions in bipartite maps is described in the following lemma. 

\begin{lemma}\label{lem:partial}
	Let $P^\mathcal{B}_i(z,u,x)$ be the generating function of planar bipartite maps with partial simple boundary of length $i$, where $z$ counts the number of edges, $u$ the root face valency divided by $2$, $x$ some distinguished labeled pattern occurrences which intersect at most pairwise. They satisfy the recursion
	\begin{align*}
		P^{\mathcal{B}}_{2i-1}(z,u,x) &= B(z,u,x)-\sum_{j=0}^{i-1}b_{i}(z,x)u^{i}-\sum_{j=1}^{i-1}P^{\mathcal{B}}_{2j-1}(z,u,x)u^{i-j}[u^{i-j}]B(z,u,x)^{2j}\\
		P^{\mathcal{B}}_{2i}(z,u,x) &= B(z,u,x)-\sum_{j=0}^{i-1}b_{i}(z,x)u^{i}-\sum_{j=0}^{i-1}P^{\mathcal{B}}_{2j}(z,u,x)u^{i-j}[u^{i-j}]B(z,u,x)^{2j+1}
	\end{align*}
    for $i\geq 1$ and initial condition $P^{\mathcal{B}}_{0}(z,u,x) = B(z,u,x)$. 
    
    In particular, the generating function of bipartite maps with partial simple boundary of length $i$ can be expressed as a polynomial in $B(z,u,x)$ and $b_0(z,x),\dots,b_i(z,x)$.
\end{lemma}

The proof is analogous to the case of general planar maps and is posponed to Section~\ref{sec:proofs}.

\begin{proposition}\label{prop:bip}
    Let $\map{p}$ be a planar bipartite map with $e$ edges, a simple boundary of length $v$ and $r$ rotational symmetries. Further, we enumerate its intersection types by $1$ to $I$, and let $e_i$ be the number of edges, $v_i$ the root face valency, $d_{i,j}$ the number of deep $j$-faces and $r_i$ the number of rotations of intersection type $i \in [I]$. Then the generating function $B(z,u,x)$ for planar bipartite maps with some distinguished labelled pattern occurrences of~$\map{p}$, intersecting at most pairwise satisfies the equation
    \begin{align}
        B(z,u,x) &= 1+zuB(z,u,x)^2 + zu\frac{B(z,1,x)-B(z,u,x)}{1-u} + x\frac{rz^{e-v+1}}{u^{\frac{v-2}{2}}}P^\mathcal{B}_{v-1}(z,u,x)\nonumber \\
        &\quad+ x^2\sum_{i=1}^I \frac{r_iz^{e_i-v_i+1}}{u^{\frac{v_i-2}{2}}}P^\mathcal{B}_{v_i-1}(z,u,x)\prod_{j\geq 1}\left([u^j]P_{j-1}^{\mathcal{B}}(z,u,x)\right)^{d_{i,j}}, \label{eq:bipartite}
    \end{align}
    where $z$ counts the number of edges, $u$ the root face valency divided by $2$, $x$ the number of distinguished occurrences of $\map{p}$.
\end{proposition}

\begin{proof}
    Again, we use the classic decomposition scheme by Tutte from the 1960s \cite{tutte_1963} which is based on deleting the root edge. The first three terms correspond to the cases where the root edge is not incident to a pattern occurrence.  

    Analogous to general maps, the fourth term corresponds to the case where the root edge is incident to a labeled pattern occurrence which does not intersect with any other labeled pattern occurrence. Note that gluing one of these $r$ rotational symmetries to the map with partial simple boundary will decrease the root face valency of the map by $v-2$ but $u$ counts the root face valency divided by $2$. Hence, we obtain a factor $u^{\frac{v-2}{2}}$.\\
    The cases where the root edge is incident to an intersection type are handled analogously.
\end{proof}

Now that the functional equation for $B(z,u,x)$ is established, we need to check whether it satisfies the assumptions of Theorem~\ref{thm:dny}. In particular, we need the following lemma. 

\begin{lemma}\label{lem:mov_bip}
	The generating functions $b_i(z,x)$ of planar bipartite maps with root face valency~$i$ are analytic functions in $z,x$ and $B(z,1,x)$ if $|z|\leq \frac{2}{15}$, $|B(z,1,x)-1| \leq \frac{5}{4}$ and $|x| \leq x_0$ for some $x_0>0$, small enough.
\end{lemma}

The proof uses a fixed point argument, similar to the one in~\cite{DrPa} and can be found in in Section~\ref{sec:proofs}. Theorem~\ref{thm:bipartite} now follows immediately.

\begin{proof}[Proof of Theorem~\ref{thm:bipartite}]
    By Lemma~\ref{lem:mov_bip} we can apply Theorem~\ref{thm:dny} to the functional equation in Proposition~\ref{prop:bip} and further deduce that $[z^n]B(z,1,x)$ satisfies the universal asymptotics given by Lemma~\ref{lem:bivar}. Consequently,
    the expectation and variance of the number of pattern occurrences of \map{p} is linear in $n$, since they can be computed from the number of maps with one or two distinguished labeled pattern occurrences divided by all bipartite maps. \\
    For higher moments, we can now apply Lemma~\ref{lem:S1}, that says the number of maps with $k$ distinguished, labeled pattern occurrences is asymptotically equal to the number of maps with $k$ distinguished labeled pattern occurrences that intersect at most pairwise if $k = \Theta(\sqrt{n})$. These maps are counted by in Proposition~\ref{prop:bip} and satisfy the asymptotics given in Lemma~\ref{lem:moments}. Hence, the factorial moments grow for $k=\Theta(\sqrt{n})$ like
    \[
        \E[(X_n)_k] \sim \mu_n^k\exp\left(\frac{k^2}{2}\frac{\sigma_n^2-\mu_n}{\mu_n^2}\right).
    \]
    Theorem~\ref{thm:gao} then yields a Gaussian central limit law.
\end{proof}

\subsection{2-connected maps}

\begin{definition}\label{def:classes} 
    \emph{2-connected} maps (also known as non-separable maps) are maps without cut-vertices. That is, they remain connected upon deletion of any vertex. We denote their generating function by
        \[
            N(z,u) = \sum_{n,j\geq 0} n_{j,n}u^jz^n,
        \]
    where $z$ marks the number of edges and $u$ the root face valency.
\end{definition}

\begin{theorem}\label{thm:clt_nonsep}
    Let $\map{p} \in \mathcal{N}$ be a fixed map and $X_n$ the number of pattern occurrences of $\map{p}$ in a random planar 2-connected map of size $n$. Then,
    \[
        \frac{X_n-\mu_n}{\sigma_n} \rightarrow \mathcal{N}(0,1)
    \]
    where $\mu_n = c_1n$ and $\sigma_n = c_2\sqrt{n}$ for some computable constants $c_1,c_2 \geq 0$  as $n\rightarrow \infty$.
\end{theorem}

2-connected planar maps always have a simple boundary. Therefore, the notion of maps with a (partial) simple boundary is redundant. However, the decomposition process is more involved than in the case of bipartite maps. Deleting the root edge may result in a non-empty sequence of 2-connected planar maps connected at newly created cut vertices, and we want to control the valency of both faces incident to the root edge. Further, we will restrict our considerations to patterns which are themselves 2-connected maps. If we allow pinch points along the boundary of the pattern, deleting the root edge results not only in a non-empty sequence of 2-connected planar maps, we also restrict cut-vertices to appear in certain places along the boundary that do not correspond to the pinch point of the pattern occurrence. However, in practice, these cases can be treated completely analogously. One just has to pay extra attention to these restrictions and, of course, justify using Lemma~\ref{lem:S1} that requires exponential bounds on the maximal vertex degree in the map. 

\begin{proposition} \label{thm:nonsep_dec}
    Let $\map{p}$ be a planar 2-connected map with $e$ edges, a boundary of length $v$ and $r$ rotational symmetries. Further, we enumerate its intersection types by $1$ to $I$, and let $e_i$ be the number of edges, $v_i$ the valency of the root face, $d_{i,j}$ the number of deep faces with valency $j$ and $r_i$ the number of rotations of intersection type $i \in [I]$. Then the generating function $N(z,u,x)$ for planar 2-connected maps with some distinguished labeled pattern occurrences of $\map{p}$, intersecting at most pairwise satisfies the equation
    \begin{align*}
        N(z,u,x) &= zu\frac{zu+\frac{uN(z,1,x)-N(z,u,x)}{1-u}}{1-\left(zu+\frac{uN(z,1,x)-N(z,u,x)}{1-u}\right)} \\
        &\quad + x\frac{rz^{e-v+1}}{u^{v-2}}F_{v-2}(z,u,N(z,u,x),n_2(z,x),\dots,n_{v-2}(z,x))\\
        &\quad + x^2\sum_{i=0}^I \frac{r_iz^{e_i-v_i+1}}{u^{v_i-2}}F_{v_i-2}(z,u,N(z,u,x),n_2(z,x),\dots,n_{v_i-2}(z,x))\prod_{j\geq 2} n_j(z,x)^{d_{i,j}}
    \end{align*}
    where all $F_v(z,u,y_0,y_1,\dots, y_{v-1}), v>0$ are polynomials.
\end{proposition}

\begin{proof}
    Analogous to the bivariate case, the first term in the equation counts all maps where the root edge is not incident to a pattern occurrence and the rest corresponds to the cases where the root edge is incident to a labeled pattern occurrence or an intersection type that we isolate.
    Next, let us point out that in case that the root edge is incident to a labeled pattern occurrence, the map without the pattern occurrence (or the occurrence of the pair of intersecting patterns) consists of a non-empty sequence of 2-connected maps or single edges as illustrated in Figure~\ref{fig:non-sep}. 

    \begin{figure}
        \centering
        \includegraphics[width=0.4\linewidth]{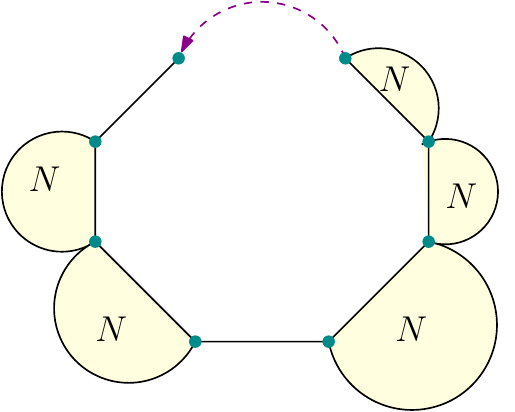}
        \caption{Decomposition of 2-connected maps}
        \label{fig:non-sep}
    \end{figure}
    
    To describe those maps, we introduce a variable $w$ that marks the number of edges incident to the pattern occurrence in addition to $z$ marking the number of edges and $u$ marking the root face valency. Each of the 2-connected maps in this sequence contribute at least one edge to the boundary of the pattern occurrence(s) and at most one less than their root face valency. Therefore, the generating function of each of those maps equals
    \begin{align*}
        \sum_{m\geq 2} n_m(z,x)u^m(w+w^2+\dots+w^{m-1})
        &=\sum_{m\geq 2} n_m(z,x)u^m\frac{w-w^m}{1-w}\\
        &=\frac{wN(z,u,x)-N(z,uw,x)}{1-w}
    \end{align*}
    Thus, the generating function of a non-empty sequence of edges or 2-connected maps that contribute $b-2$ edges to the boundary of a pattern occurrence or an intersection type is described by
    \begin{align*}
        [w^{b-2}] \frac{zu^2w+\frac{wN(z,u,x)-N(z,uw,x)}{1-w}}{1-\left(zu^2w+\frac{wN(z,u,x)-N(z,uw,x)}{1-w}\right)}
    \end{align*}
    which obviously is a polynomial $F_{b-2}(z,u,y_0,y_1,\dots, y_{b-3})$ in $z,u,$ $y_0=N(z,u,x)$ and $y_1 =n_2(z,x), \dots, y_3 = n_{b-2}(z,x)$.

    Subsequently, we add $e-v+1$ edges of the pattern occurrence which decreases the root face valency by $v-2$. Further, we multiply the expression by the number of rotational symmetries of the pattern to account for all possible ways of attaching the pattern occurrence to the sequence of 2-connected maps and mark the term with $x$.

    In the case where the root edge is incident to an intersection type, the above argument holds as well. Additionally, the deep faces are filled by 2-connected maps with suitable root face valency (which have a simple boundary by definition).
\end{proof}

\begin{lemma}\label{lem:nonsep_ana}
    The generating functions $n_i(z,x)$ of 2-connected planar maps with root face valency~$i$ are analytic functions in $z,x$ and $N(z,1,x)$ if $|z|\leq \frac{1}{6}$, $|N(z,1,x)| \leq \frac{1}{18}$ and $|x| \leq x_0$ for some $x_0>0$, small enough.
\end{lemma}

\begin{proof}
    First, we reformulate the equation to a polynomial expression
    \begin{align*}
        N(z,u,x) &= \left(zu-zu^2+uN(z,1,x)-N(z,u,x)\right)(zu+N(z,u,x)) \\
        &\qquad + xF(z,u,N(z,1,x),n_2(z,x),\dots, n_{2v}(z,x)).
    \end{align*}
    Now, the proof is analogous to the one of bipartite maps. Since the singularity of $N(z,1,0)$ is at $z=\frac{4}{27}$, and $N(\frac{4}{27},1,0)=\frac{1}{27}$ and $u(\frac{4}{27})=\frac{3}{2}, N(\frac{4}{27},\frac{3}{2},0) = \frac{1}{9}$, we can show convergence for $|z|\leq \frac{1}{6}, |N(z,1,x)|\leq \frac{1}{18}, |u|\leq \frac{7}{4}$ and $|x|<x_0$, where $x_0$ is determined by the maximum admitted in this range of values by the terms in the functional equation that appear with a factor of $x$.
\end{proof}

\begin{proof}[Proof of Theorem~\ref{thm:clt_nonsep}]
The proof is absolutely analogous to general and bipartite maps.
\end{proof}

\section{Proofs}\label{sec:proofs}

\subsection{Proof of Lemma \ref{lem:S1}}

First, we distinguish maps according to the number of patterns that appear. Let $S_1$ be the set of maps in $\mathcal{C}$ with less than $\mu_n/2$ pattern occurrences. Then the contribution of maps in $S_1$ with $k$ labeled patterns to $m_{n,k}$ is bounded by
\[
    (\mu_n/2)^k|S_1| \leq (\mu_n/2)^km_n
\] 
since there are at most $(\mu_n/2)_k$ ways to mark $k$ patterns in a map of the set $S_1$.

Next, we have a closer look at the contribution of the set maps $S_2$ with more than $\mu_n/2$ pattern occurrences. Let \map{m} in $S_2$ with $p$ occurrences of the pattern \map{p} and let $L(\map{m})$ be the set of maps consisting of the map \map{m} with $k$ labeled pattern occurrences. 

Further, let $c(n)$ be the number of ways a pattern occurrence can intersect with another. Notice, that 
\begin{enumerate}[(a)]
    \item if \map{p} has a simple boundary, $c(n)$ is bounded by a constant given by $vf^2$ , where $f$ is the number of interior faces of \map{p} and $v$ is the maximal face valency of the interior faces. This follows because a face $F$ in \map{m} can intersect with at most $f$ different faces of an occurrence of \map{p} and the intersection can be rotated in at most $v$ ways.
    \item if \map{p} has a boundary with $m$ pinch points and $\ell$ is the maximal number of components that are produced by deleting one of these pinch points, then $c(n)$ is bounded by $\max\{cf^2, \binom{\Delta(\map{m})}{2\ell}\}$, where $\Delta(\map{m})$ is the maximal vertex degree in \map{m}. This follows because two pattern occurrences can intersect by sharing the same pinch point on their boundary and there are $\binom{\Delta(\map{m})}{2\ell}$ ways to choose $2\ell$ corners in which the components of the two occurrences can be located.
\end{enumerate}

Further, since there are $(p)_k$ ways to label $k$ pattern occurrences the size of $L(\map{m})$ is exactly $(p)_k$.
If we estimate from above the number of ways to label $k$ patterns where at least one labeled occurrence intersects at least two other occurrences, there are at most $p$ choices for the pattern occurrence which intersects two others and $k$ choices for its label. There are further at most $c(n)^2$ choices for the two pattern occurrences that the labeled occurrence intersects and $(k-1)$ and $(k-2)$ choices for their labels respectively. The rest of the labeled pattern occurrences can be chosen freely and we obtain that there are in total at most
\[
    (pk)\cdot \left(c(k-1)\right)\cdot \left(c(k-2)\right)\cdot (p-3)_{(k-3)}
\]
ways to label $k$ pattern occurrences such that there exists at least one labeled occurrence that intersects two other labeled occurrences. Comparing this number to the size of $L(\map{m})$ yields
\[
    \frac{c^2k(k-1)(k-2)p(p-3)(p-4)\cdots (p-k+1)}{p(p-1)\cdots (p-k+1)} \leq \frac{c^2k^3}{(p-2)^2}.
\]

At this point, we make another case distinction. Let $S_{2a} \subseteq S_2$ be the subset of $S_2$ where $\Delta(\map{m}) < \log(n)^2$ and $S_{2b} \subseteq S_2$ the subset where $\Delta(\map{m}) \geq \log(n)^2$.
\begin{itemize}
    \item 

If $\map{m} \in S_{2a}$, then $c(n) < \log(n)^{4\ell}$ and the above is bounded by $\frac{C\log(n)^{8\ell}}{\sqrt{n}}$ for some constant $C$, since $p>\mu_n/2$ and $\mu_n = \Theta(n)$.

\item If $\map{m} \in S_{2b}$, we can estimate the above by $k^3$, since it has to hold $c\leq p$. However, the probability that $\Delta(m) > \log(n)^2$ is bounded exponentially. That is, $S_{2b} \leq e^{-\delta \log(n)^2}m_n$.
\end{itemize}

By using the rough estimates $|S_1| < m_n$ and $\sum_{\ell \geq \mu_n/2} (\ell)_k\, m_{n,\ell} < m_{n,k}^\circ$, the considerations above sum up to
\[
     m_{n,k}^\circ \leq m_{n,k}^{\circ,\times} + \left(\frac{\mu_n}{2}\right)^k m_{n} + O\left(\left(\frac{C\log(n)^{8\ell}}{\sqrt{n}}+\frac{k^3}{n^{
     \delta \log n
     }}\right)m_{n,k}^\circ\right).
\]

\subsection{Proof of Lemma~\ref{lem:mov_bi}}

\begin{proof}
	We observe that the generating functions $b_{i}(z,x)$ satisfy a recursion of the form
	\begin{align*}
		b_{0}(z,x) &= 1\\
		b_{i}(z,x) &= z\sum_{j=0}^{i-1}b_{j}(z,x)b_{i-j-1}(z,x)+z(B(z,1,x)-\sum_{j =0}^{i-1}b_{j}(z,x)) \\
        &\qquad +x[u^i] P(z,u,B(z,u,x),b_0(z,x),\dots,b_k(z,x)) \\
        &\qquad + x^2[u^i]Q(z,u,B(z,u,x),b_0(z,x),\dots,b_\ell(z,x))
	\end{align*}
    by Proposition~\ref{thm:bipartite}, where we denote
    \begin{align*}
        p_i(z,b_{i-k'}(z,x), b_{i-k'+1}(z,x),\dots, &b_{i}(z,x),b_0(z,x),\dots,b_k(z,x)) \\
        &:= [u^i] P(z,u,B(z,u,x),b_0(z,x),\dots,b_k(z,x))\\
        q_i(z,b_{i-\ell'}(z,x), b_{i-\ell'+1}(z,x),\dots, &b_{i}(z,x),b_0(z,x),\dots,b_\ell(z,x)) \\
        &:= [u^i] Q(z,u,B(z,u,x),b_0(z,x),\dots,b_\ell(z,x)).
    \end{align*}
    Now, we know that equation~\eqref{eq:bipartite} is satisfied for small $z$ and $u$ close to~$1$ and yields analytic solutions $B(z,1,x)$ and $B(z,u,x)$. We also know that for $x=0$, the radius of convergence of $B(z,1,0)$ is $1/8$ and $B(1/8,1,0)=5/4$. We therefore may assume $Y = B(z,1,x)$ is given and let $x_0$ be small enough such that for all $|x|\leq x_0$, we have $|Y| \leq 11/8$. Further, we define the operator $T: \ell^1 (\C) \rightarrow \ell^1(\C)$ which maps ${\bf y} = (y_0,y_1,y_2, \dots) \in \ell^1(\C)$ for fixed $|x| \leq x_0$ and $|z| \leq \frac{2}{15}$ to
    \begin{align*}
		T({\bf y})_i &= z\sum_{j=0}^{i-1}y_{j}y_{i-j-1}+z\left(Y-\sum_{j =0}^{i-1}y_{j} \right) \\
        &\qquad +xp_i(z,y_{i-k'},\dots, y_i,y_0,\dots,y_k) + x^2q_i(z,y_{i-\ell'},\dots,y_i,y_0,\dots,y_\ell)
	\end{align*}
    Since $P$ and $Q$ are polynomials there exists a maximum
    \[
        M_1 = \max_{|z|\leq \frac{2}{15}, |x|\leq x_0, |{\bf y}|_1\leq 2} \left\{ \left|P(z,1,|{\bf y}|_1,y_0,\dots,y_{k})+xQ(z,1,|{\bf y}|_1,y_0,\dots,y_{\ell})\right|  \,\right\},
    \]
    and
    \[
        M_2 = \max_{|z|\leq \frac{2}{15}, |x|\leq x_0, |{\bf y}|_1\leq 2} \left\{ \Big|\partial_{{\bf y},y_0,\dots,y_{\ell}}\Big(P(z,1,|{\bf y}|_1,y_0,\dots,y_{k})+xQ(z,1,|{\bf y}|_1,y_0,\dots,y_{\ell})\Big)\Big|\, \right\},
    \]
    such that we can further restrict $|x| < \min \left\{x_0,\frac{1}{60M1},\frac{1}{15M_2}\right\} := M$.

    Now, we want to show that if we look at a sequence $({\bf y_n})_{n\geq 0}$ with $\bf y_0 = 0$ and ${\bf y_n} = T({\bf y_{n-1}})$, it converges uniformly to a fixed point ${\bf y} = (y_0,y_1,\dots)$ with $y_i = b_i(z,x,Y)$ by the Banach fixed-point theorem.

    So, first we show that $T$ maps elements with $|{\bf y}| \leq 2$ to elements $|T({\bf y})| \leq 2$. Note that for $|z|\leq \frac{2}{15}$, $|x|\leq M$, we can simply estimate the norm by
    \begin{align*}
        \mid T({\bf y_n})-1\mid_1  &\leq |z| |{\bf y_{n-1}}|_1^2+|z||Y| + |z||{\bf y_{n-1}}|_1 + xM_1 \\
        &\leq \frac{8}{15}+\frac{2}{15}\frac{11}{8}+\frac{4}{15}+\frac{1}{60} \\ &\leq \frac{59}{60}+\frac{1}{60} 
    \end{align*}
    and thus $|T({\bf y_n})| \leq 2$. Further, we show that the map is a contraction. That is, 
    \begin{align*}
        \mid T({\bf y_n})-T({\bf y_{n-1}})\mid_1  &\leq |z|\left(|\bf {y_n}|+|{\bf y_{n-1}}|\right) |{\bf y_{n}}-{\bf y_{n-1}}|_1 + |z||{\bf y_n}-{\bf y_{n-1}}|_1 + xM_2|{\bf y_n}-{\bf y_{n-1}}|_1 \\
        &\leq \left(\frac{8}{15}+\frac{11}{60}+\frac{1}{15}+\frac{1}{15}\right) |{\bf y_n}-{\bf y_{n-1}}|_1\\
        &<|{\bf y_n}-{\bf y_{n-1}}|_1
    \end{align*}
    and therefore, we can conclude that for $|z|<\frac{2}{15}$ the above map gives us a sequence converging to a fixed point $\bf y$ with $y_i=b_i(z,x,Y)$ and therefore, we can express all solutions as $b_i(z,x,M(z,1,x))$ for $x$ small enough.
\end{proof}

\subsection{Proof of Lemma~\ref{lem:bivar}} 
\begin{proof}
    By Theorem~2 in~\cite{DrmotaNoyYu}, the coefficients of $B(z,1,x)$ asymptotically grow like
    \[
        [z^n]B(z,1,x) = c(x)n^{-5/2}\rho(x)^n\left(1+O\left(\frac{1}{n}\right)\right),
    \]
    where $\rho(x)^{-1}$ is the location of the dominant singularity of $B(z,1,x)$, depending on the value $x$. Since the dominant singularity at $x=0$ is unique, there will also be a unique dominant singularity for small  values of $x$. Further, we note that the asymptotic behavior is uniformly for $x$ small enough.\\ Now, we rewrite
    \[
        [z^n]B(z,1,x) = F_n(x) =  n^{\alpha}e^{nf(x)+g(x)+O(\frac{1}{n})}
    \]
    where $\alpha = -\frac{5}{2}, f(x) = \log \rho(x)$ and $g(x) = \log c(x)$ and use the Cauchy integral:
    \[
        [x^k]\frac{F_n(x)}{n^{\alpha}} = \frac{1}{2\pi i} \int_\gamma \frac {e^{nf(x) + g(x)+O(\frac{1}{n})}}{x^{k+1}}\, dx,
    \]
    where $\gamma$ is a contour that surrounds $x=0$. 
    In particular, we choose $\gamma$ to be a cycle with
    radius $\frac{k}{nf'(0)}$ (where we expect a saddle point). That is, $x = r e^{i\varphi}$ with $r = \frac{k}{nf'(0)}$ and $dx = ire^{i\varphi}d\varphi$, and therefore,
    \begin{align*}
        [x^k]\frac{F_n(x)}{n^{\alpha}} &= 
        \frac{1}{2\pi}r^{k}\int_{-\pi}^{\pi} \frac{e^{nf(re^{i\varphi})+g(re^{i\varphi})+O(\frac{1}{n})}}{e^{ik\varphi}}d\varphi.
    \end{align*}
    Since $f(x) = \log \rho(x)$ and $g(x) = \log c(x)$ are analytic, we can use their Taylor expansions at $re^{i\varphi} = 0$ and obtain
    \begin{align*}
        [x^k]\frac{F_n(x)}{n^{\alpha}} &= \frac{1}{2\pi}r^{k} \int_{-\pi}^\pi e^{\sum_{\ell=0}^3 \left(nf^{(\ell)}(0)+g^{(\ell)}(0)\right)\frac{(r e^{i\varphi})^\ell
        }{\ell!}+O(nr^4+\frac{1}{n})} e^{-ik\varphi}\, d\varphi \\
        & = \frac{1}{2\pi}   \left(\frac{f'(0)n}{ k}\right)^{k}
        e^{nf(0)+g(0)}
        \int_{-\pi}^\pi  e^{k(e^{i\varphi} - i\varphi) + \frac{k^2}{2n} \frac{f''(0)}{(f'(0))^2}  e^{2i\varphi} 
         + O\left(\frac{k}{n}\right)} \, d\varphi.
    \end{align*}
    Next, we are going to factor out the constant terms of the exponential functions in the exponent of the integrand. That is,
    \begin{align}
        &(f'(0)n)^k  e^{nf(0)+g(0)}\frac{1}{2\pi k^{k}}  
        \int_{-\pi}^\pi  e^{k(e^{i\varphi} - i\varphi) + \frac{k^2}{2n} \frac{f''(0)}{f'(0)^2} e^{2i\varphi} 
         + O\left(\frac{k}{n}\right)} \, d\varphi \nonumber\\
        & = (f'(0)n)^k e^{nf(0)+g(0)+\frac{k^2}{2n} \frac{f''(0)}{f'(0)^2}}  \frac{e^k}{2\pi k^{k}}  
        \int_{-\pi}^\pi  e^{k(e^{i\varphi} - 1 - i\varphi)} 
        e^{+\frac{k^2}{2n} \frac{f''(0)}{f'(0)^2} + O\left(\frac{k}{n}\right)} \, d\varphi. \label{eq:fnk}
    \end{align}
    In the last step, we already separated the integrand in two factors, where the second is a bounded function as $n\rightarrow\infty$. Thus, we expand the latter in the integral as
    \begin{align*}
        &  
        \int_{-\pi}^\pi  e^{k(e^{i\varphi} - 1 - i\varphi) + \frac{k^2}{2n} \frac{f''(0)}{(f'(0))^2}  (e^{2i\varphi} -1) 
        + O(\frac{k}{n})} \, d\varphi \\
        &=   
        \int_{-\pi}^\pi  e^{k(e^{i\varphi} - 1 - i\varphi) }\sum_{j\geq 0} \frac{1}{j!}\left(\frac{k^2}{2n} \frac{f''(0)}{(f'(0))^2}  (e^{2i\varphi} -1) + O\left(\frac{k}{n}\right)\right) ^j\, d\varphi\\
        &=\int_{-\pi}^\pi  e^{k(e^{i\varphi} - 1 - i\varphi)}d\varphi +
        \int_{-\pi}^\pi e^{k(e^{i\varphi} - 1 - i\varphi)}\sum_{j\geq 1}\frac{1}{j!}\left(\frac{k^2}{2n} \frac{f''(0)}{(f'(0))^2}  (e^{2i\varphi} -1) + O\left(\frac{k}{n}\right)\right) ^j d\varphi.
    \end{align*}
    Now, we may exchange the integral and the sum of the exponential function and, as $k^2/n=O(1)$, proceed to evaluate 
    
    \begin{align}\label{eq:I1}
    \int_{-\pi}^\pi  &e^{k(e^{i\varphi} - 1 - i\varphi)}d\varphi +O \left(
        \int_{-\pi}^\pi  e^{k(e^{i\varphi} - 1 - i\varphi)} \left(\frac{k^2}{2n} \frac{f''(0)}{(f'(0))^2} (e^{2i\varphi} -1) + \frac{k}{n}\right)d\varphi\right).
    \end{align}
    The first integral \eqref{eq:I1} can be evaluated by substituting back $z=e^{i\varphi}$ on a circle $\gamma'$ around $z=0$ with $d\varphi = \frac{dz}{iz}$ and interpreting the Cauchy integral as a coefficient extraction of the integrand. In particular, we have
    \begin{align*}
        \int_{-\pi}^\pi  e^{k(e^{i\varphi} - 1 - i\varphi)} d\varphi &= \frac{1}{ie^{k}}\int_{\gamma'}  \frac{e^{kz}}{z^{k+1}}dz = \frac{2i\pi}{ie^{k}}
        [z^k]e^{kz} = \frac{2\pi k^k}{e^{k}k!}
        .
    \end{align*}
 
    For the second integral in \eqref{eq:I1}, we use a variation of Laplace's method. We have
    \begin{align*}
        &\Re \left(k(e^{i\varphi} - 1 - i\varphi)\right) \le -k c_1\varphi^2
    \end{align*}
    for some constant $c_1> 0$, $|\varphi| \le \pi$ and $n$ large enough. 
    Therefore, we can bound the value of the integral by
    \begin{align*}
        \int_{-\pi}^\pi  e^{k(e^{i\varphi} - 1 - i\varphi)} \left(\frac{k^2}{2n} \frac{f''(0)}{f'(0)^2} (e^{2i\varphi} -1) + \frac{k}{n}\right)d\varphi &\leq  \int_{-\pi}^{\pi} e^{-c_1k \varphi^2} \left(c_2 |\varphi| \frac{k^2}{n} + \frac{k}{n}  
       \right) d\varphi       \\
        &= c_2\frac{k}{n}(1-e^{-c_1k\pi^2}) + \frac{\sqrt{\pi k}}{n}+ O\left(\frac{1}{n}\right),
    \end{align*}
    where $c_2>0$ is another constant factor.
    Putting everything together, we have 
    \begin{align*}
        \eqref{eq:I1} 
        &= \frac{2\pi k^k}{e^k k!} + O\left(\frac{k}{n}+\frac{\sqrt{k}}{n}\right)= \frac{2\pi k^k}{e^k k!}\left(1 + O\left(\frac{1}{n^{\frac{1}{4}}}\right)\right)
    \end{align*}
    since $\frac{2\pi k^k}{e^k k!} \leq \sqrt{\frac{2\pi}{k}}$ and consequently,
    \begin{align*}
        [x^k]\frac{F_n(x)}{n^\alpha} & = (f'(0)n)^k  e^{nf(0)+g(0)} \frac{e^k}{2\pi k^{k}}  
        \int_{-\pi}^\pi  e^{k(e^{i\varphi} - 1 - i\varphi) + \frac{k^2}{2n} \frac{f''(0)}{f'(0)^2} (e^{2i\varphi} -1) + O\left(\frac{k}{n}\right)} \, d\varphi\\
        &= \frac{1}{k!}(f'(0)n)^k  e^{nf(0)+g(0)+\frac{k^2}{2n} \frac{f''(0)}{f'(0)^2}}\left(1+O\left(\frac{1}{n^{\frac 14}}\right)\right).
    \end{align*}   
    Computing the derivatives of $f(x)$ in terms of $\rho(x)$ completes the proof of the theorem.
\end{proof}

\subsection{Proof of Lemma~\ref{lem:moments}}
\begin{proof}
    For the expectation, we simply compute
    \begin{align*}
        \E(X_n) &= \frac{[z^n]\partial_x F(z,0,x)|_{x=0}}{[z^n]F(z,0,x)|_{x=0}} \sim \frac{c'(0)n^{-\frac{5}{2}}\rho(0)^n+c(0)n^{-\frac{3}{2}}\rho(0)^{n-1}\rho'(0)}{c(0)n^{-\frac{5}{2}}\rho(0)^n}\sim \frac{\rho'(0)}{\rho(0)}n + \frac{c'(0)}{c(0)}
    \end{align*}
    and for the variance we then obtain
    \begin{align*}
        \V(X_n) &= \E(X_n(X_n-1)) + \E(X_n) - \E(X_n)^2\\
        &= \frac{[z^n]\partial_{x}^2 F(z,0,x)|_{x=0}}{[z^n]F(z,0,x)|_{x=0}} + \E(X_n) - \E(X_n)^2\\
        &\sim \frac{c''(0)n^{-\frac{5}{2}}\rho(0)^n+2c'(0)n^{-\frac{3}{2}}\rho(0)^{n-1}\rho'(0)+c(0)n^{-\frac{3}{2}}(n-1)\rho(0)^{n-2}\rho'(0)^2}{c(0)n^{-\frac{5}{2}}\rho(0)^n}\\
        &\quad +\frac{c(0)n^{-\frac{3}{2}}\rho(0)^{n-1}\rho''(0)}{c(0)n^{-\frac{5}{2}}\rho(0)^n}+ \E(X_n) - \E(X_n)^2\\
        &\sim \frac{\rho'(0)^2}{\rho(0)^2}(n^2-n) + \frac{2c'(0)\rho'(0)+c(0)\rho''(0)}{c(0)\rho(0)}n + \frac{\rho'(0)}{\rho(0)}n - \left(\frac{\rho'(0)}{\rho(0)}n + \frac{c'(0)}{c(0)}\right)^2\\
        &\sim \frac{\rho''(0)\rho(0)+\rho'(0)\rho(0)-\rho'(0)^2}{\rho(0)^2}n
    \end{align*}
    The asymptotics for the higher moments follow directly from Lemma~\ref{lem:bivar} and subsequent division by $[z^n] F(z,0,0)$.
\end{proof}

\subsection{Proof of Lemma~\ref{lem:partial}}
\begin{proof}
	In order to enumerate the number of maps with partial simple boundary of length $\ell$, we first consider all maps with root face valency at least $\ell+1$. The generating function of maps with root face valency greater than $\ell$ is described by
	\[
		B(z,u,x)-\sum_{j=0}^{\ell}b_{j}(z,x)u^{j}.
	\]
	If a map with root face valency larger than $\ell$ does not have a partial simple boundary of length $\ell$, we can decompose the first $\ell$ edges along the boundary into boundaries of bipartite maps with root face valency smaller than $\ell$ that are attached at the vertices of a partial simple boundary path that contains the root edge and whose length is at least $1$ and at most $\ell-2$.
    Since each cycle on a bipartite map must be of even length, the length of the partial simple boundary path containing the root edge has the same parity as~$\ell$ and therefore the generating function of bipartite maps without a partial simple boundary of odd length $\ell = 2i-1$ is 
	\[
		\sum_{j=1}^{i-1}P^{\mathcal{B}}_{2j-1}(z,u,x)u^{i-j}[u^{i-j}]B(z,u,x)^{2j}
	\]
	and that of maps without a partial simple boundary of even length $\ell = 2i$ is 
	\[
		\sum_{j=0}^{i-1}P^{\mathcal{B}}_{2j}(z,u,x)u^{i-j}[u^{i-j}]B(z,u,x)^{2j+1}.
	\]
    The statement about $p_{i,j}(z,x)$ follows directly from the fact that 
    \[
        p_{i,j}(z,x) = [u^j]P_{i}(z,u,x)
    \]
    and the recursion above.
\end{proof}

\section{Conclusion}

We presented a simplified proof of asymptotic normality for pattern counts in random planar maps based on direct asymptotic analysis of catalytic functional equations. In contrast to the approach in \cite{patternocc}, which reduces the problem to face counts, the present method computes the asymptotics of factorial moments directly from a bivariate generating function with one catalytic variable. 
The method applies in particular to bipartite maps and extends without major changes to 2-connected maps, where the reduction  cannot be applied directly. The proof is formulated in a way that also covers patterns with non-simple boundary whenever suitable bounds on the maximal vertex degree are available.
It is natural to ask whether the same strategy can be used for further classes of maps, such as triangulations or maps with prescribed face degrees. In these cases the corresponding functional equations typically are singularly perturbed discrete differential equations, and the universal asymptotic results used here are no longer directly applicable. Partial results for such equations have been obtained in recent work of the author in collaboration with Michael Drmota, but a general theory sufficient for treating pattern counts in these map classes is not yet available.


\bibliography{Bib_maps.bib}{}
\bibliographystyle{plain}

\end{document}